\documentclass[11pt]{article}

\usepackage[utf8]{inputenc}
\usepackage{amsmath,amssymb,amsthm,mathtools}
\usepackage[a4paper,margin=30mm]{geometry}
\usepackage{microtype}
\usepackage{booktabs}
\usepackage{enumitem}
\usepackage[numbers,sort&compress]{natbib}
\usepackage[hidelinks]{hyperref}
\usepackage{orcidlink}
\hypersetup{
  pdftitle={A Proof of the B-Free Graphs Conjecture},
  pdfauthor={Domenico Frijio},
  pdfsubject={Extremal and structural graph theory},
  pdfkeywords={stability number, clique number, forbidden induced subgraph, sum-perfect graph}
}

\newtheorem{theorem}{Theorem}[section]
\newtheorem{lemma}[theorem]{Lemma}
\newtheorem{proposition}[theorem]{Proposition}
\newtheorem{corollary}[theorem]{Corollary}
\numberwithin{equation}{section}
\theoremstyle{remark}

\newcommand{\cB}{\mathcal B}
\newcommand{\cI}{\mathcal I}

\newcommand{\core}{\operatorname{core}}
\newcommand{\corona}{\operatorname{corona}}
\newcommand{\e}{\mathrm e}

\title{\bfseries A Proof of the B-Free Graphs Conjecture}
\author{%
  Domenico Frijio\,\orcidlink{0009-0005-2747-3961}\\[-1mm]
  \small Horizon Research\\[-1mm]
  \small \href{mailto:dfrijio@horizonrsc.com}{dfrijio@horizonrsc.com}}
\date{22 August 2026}

\begin{document}
\maketitle

\begin{abstract}
Let $\cB$ be the class consisting of the six-vertex bipartite graphs that possess a perfect matching and their complements. It is proved that every $\cB$-free graph $G$ satisfies $\alpha(G)+\omega(G)\ge |V(G)|-1$. This establishes Conjecture 3.1 of Litjens, Polak and Sivaraman (B-Free Graphs Conjecture). For a smallest counterexample, Hall-type exchange arguments show that two maximum stable sets, and likewise two maximum cliques, differ in at most two vertices. A core--corona matching lemma then forces $|\alpha(G)-\omega(G)|\le 2$. Double counting between suitably dense and sparse vertices reduces the problem to twenty-one binary feasibility systems on at most fourteen vertices. Their infeasibility is verified by two independent exact encodings, with a separate exhaustive validation of the forbidden-family constraints.
\end{abstract}



\section{Introduction}

All graphs in this paper are finite and simple.  For a graph $G$, write
$n(G)=|V(G)|$, and let $\alpha(G)$ and $\omega(G)$ denote its stability and
clique numbers.  Litjens, Polak and Sivaraman called $G$ \emph{sum-perfect} if
every induced subgraph $H$ satisfies
\(
\alpha(H)+\omega(H)\ge n(H)
\), and characterized this hereditary class by twenty-seven forbidden induced
subgraphs \cite{LitjensPolakSivaraman2019}.  Their first twenty-four
nontrivial obstructions, denoted $H_2,\ldots,H_{25}$, admit the following
invariant description:
\[
 \begin{aligned}
 \cB={}&\{H:n(H)=6,\ H\text{ is bipartite and has a perfect matching}\}\\
       &{}\cup\{\overline H:H\text{ belongs to the first set}\}.
 \end{aligned}
\]
Equivalently, a graph contains a member of $\cB$ if it has either two disjoint
stable triples with three independent cross-edges, or two disjoint triangles
with three independent cross-nonedges.  Here and below, containment means
induced containment.

The conjecture considered in \cite[Conjecture~3.1]{LitjensPolakSivaraman2019}
asserts that excluding $\cB$ permits a deficit of at most one from the
sum-perfect inequality. In this study, Conjecture~3.1 is referred to as the \emph{B-Free Graphs Conjecture}.  We prove it.

\begin{theorem}\label{thm:main}
If $G$ is $\cB$-free, then
\[
  \alpha(G)+\omega(G)\ge n(G)-1.
\]
\end{theorem}

The result belongs to the broader study of complementary graph parameters,
which begins with the Nordhaus--Gaddum inequalities
\cite{NordhausGaddum1956,AouchicheHansen2013} and is closely connected with
perfect graphs \cite{Lovasz1972,ChudnovskyRobertsonSeymourThomas2006}.
The hereditary form of such inequalities has recently been studied for the
chromatic number and its complementary counterpart
\cite{SivaramanWhitman2025}.  Sum-perfect graphs also meet the classical theory
of split and threshold graphs
\cite{FoldesHammer1977Split,FoldesHammer1977Dilworth,
HammerSimeone1981,HammerIbarakiSimeone1981}; useful general accounts are
\cite{Golumbic2004,BrandstadtLeSpinrad1999}.  Structural and optimization
results for perfect and weakly chordal graph classes provide further context
\cite{GrotschelLovaszSchrijver1984,HaywardHoangMaffray1989}.

The proof has two parts.  Sections~\ref{sec:exchange} and \ref{sec:reduction}
give structural reductions that bound a smallest counterexample by fourteen
vertices.  Section~\ref{sec:finite} formulates the remaining twenty-one cases
as explicit finite $0$--$1$ systems and records their infeasibility.

\section{Exchange structure of a smallest counterexample}\label{sec:exchange}

For two sets $A,B$, write $A\triangle B$ for their symmetric difference.
The following elementary exchange lemma is the point at which the special
form of $\cB$ enters.

\begin{lemma}\label{lem:exchange-matching}
Let $A$ and $B$ be maximum stable sets of a graph $G$.  There is a matching
from $A\setminus B$ onto $B\setminus A$ in the bipartite graph induced by
these two sets.
\end{lemma}

\begin{proof}
The two sides have the same cardinality.  If Hall's condition
\cite{Hall1935} fails, there is $X\subseteq A\setminus B$ with
$|N(X)\cap(B\setminus A)|<|X|$.  Since $A$ is stable, $X$ has no neighbour in
$A\cap B$.  Consequently
\[
  \bigl(B\setminus N(X)\bigr)\cup X
\]
is stable and has more than $|B|$ vertices, a contradiction.
\end{proof}

\begin{corollary}\label{cor:diameter}
If $G$ is $\cB$-free and $A,B$ are maximum stable sets, then
$|A\setminus B|=|B\setminus A|\le 2$.  Dually, the same assertion holds for
two maximum cliques.
\end{corollary}

\begin{proof}
If $|A\setminus B|\ge3$, choose three edges of the matching in
Lemma~\ref{lem:exchange-matching}.  Their endpoints form two stable triples
whose cross-edges contain a perfect matching, and hence induce a member of
$\cB$.  The clique statement follows by complementation.
\end{proof}

Let $G$ now be a smallest counterexample to Theorem~\ref{thm:main}, and put
\[
  n=n(G),\qquad a=\alpha(G),\qquad w=\omega(G).
\]
Every proper induced subgraph satisfies the theorem.  Thus, for every
$v\in V(G)$,
\[
 a+w=n-2,
 \qquad
 \alpha(G-v)=a,
 \qquad
 \omega(G-v)=w.                                      \tag{2.1}\label{eq:minimal}
\]
Indeed, the counterexample inequality gives $a+w\le n-2$, whereas the theorem
for $G-v$ gives
$\alpha(G-v)+\omega(G-v)\ge n-2$; both terms on the left are bounded above by
$a$ and $w$, respectively.  Hence equality holds throughout.
In particular, no vertex belongs to all maximum stable sets, and no vertex
belongs to all maximum cliques.

Let $\cI(G)$ be the family of maximum stable sets and define
\[
 \core_\alpha(G)=\bigcap_{I\in\cI(G)}I,
 \qquad
 \corona_\alpha(G)=\bigcup_{I\in\cI(G)}I.
\]
Define $\core_\omega(G)$ and $\corona_\omega(G)$ analogously for maximum
cliques.  Equation~\eqref{eq:minimal} gives
\[
 \core_\alpha(G)=\core_\omega(G)=\varnothing.        \tag{2.2}\label{eq:emptycore}
\]

We use the set-and-collection lemma of Levit and Mandrescu
\cite{LevitMandrescu2014}; see also
\cite{JardenLevitMandrescu2019}.  In the form needed here, it states that, for
a maximum stable set $S$, there is a matching from
$S\setminus\core_\alpha(G)$ into
$\corona_\alpha(G)\setminus S$.  This result refines the usual Hall-type
exchange principle; matching theory in graphs is developed more generally in
\cite{Edmonds1965}.

\begin{proposition}\label{prop:corona}
For every maximum stable set $S$ of $G$, there are at least $a$ distinct
vertices $v\notin S$ such that
\[
 |N(v)\cap S|\le2,
\]
and these vertices can be matched to $S$ by edges.  Dually, for every maximum
clique $K$, there are at least $w$ distinct vertices $v\notin K$ having at
most two nonneighbours in $K$, and they can be matched to $K$ by nonedges.
\end{proposition}

\begin{proof}
By \eqref{eq:emptycore} and the set-and-collection lemma, a matching saturates
$S$ into $\corona_\alpha(G)\setminus S$.  Let $v$ be one of its targets and
choose a maximum stable set $T$ containing $v$.  The vertex $v$ has no
neighbour in $S\cap T$, while Corollary~\ref{cor:diameter} gives
$|S\setminus T|\le2$.  Hence $|N(v)\cap S|\le2$.  Apply the same argument to
$\overline G$ for the dual statement.
\end{proof}

\begin{corollary}\label{cor:balance}
One has $2a\le n$, $2w\le n$, and consequently
\[
 |a-w|\le2.                                           \tag{2.3}\label{eq:balance}
\]
\end{corollary}

\begin{proof}
The first matching in Proposition~\ref{prop:corona} places $a$ distinct
vertices outside a set of size $a$, so $2a\le n$; the dual matching gives
$2w\le n$.  Substitute $n=a+w+2$ from \eqref{eq:minimal}.
\end{proof}

\section{Reduction to at most fourteen vertices}\label{sec:reduction}

Replacing $G$ by its complement if necessary, assume $a\ge w$.  By
\eqref{eq:balance}, write
\[
 d=a-w\in\{0,1,2\}.                                  \tag{3.1}\label{eq:d}
\]
Fix a maximum stable set $S$ and a maximum clique $K$.  Since
$|S\cap K|\le1$, there are two cases.

\begin{lemma}\label{lem:disjoint}
If $S\cap K=\varnothing$, then $a\le6$ for $d=0,1$, and $a\le7$ for $d=2$.
In particular, $n\le14$.
\end{lemma}

\begin{proof}
Here $|V(G)\setminus(S\cup K)|=2$.  Proposition~\ref{prop:corona} therefore
gives a set $L\subseteq K$ of cardinality $q\ge a-2$ such that every vertex
of $L$ has at most two neighbours in $S$.  Dually, it gives a set
$D\subseteq S$ of cardinality $p\ge w-2$ such that every vertex of $D$ has at
most two nonneighbours in $K$.  Counting edges between $D$ and $L$ yields
\[
 p(q-2)\le \e(D,L)\le2q.                              \tag{3.2}\label{eq:count0}
\]
For $a\ge5$, the function $x\mapsto 2x/(x-2)$ is decreasing on $x>2$.
Since $q\ge a-2$, inequality \eqref{eq:count0} implies
\[
 (w-2)(a-4)\le2(a-2).                                \tag{3.3}\label{eq:bound0}
\]
Substituting $w=a-d$ in \eqref{eq:bound0} gives, respectively,
\[
 (a-2)(a-4)\le2(a-2),\quad
 (a-3)(a-4)\le2(a-2),\quad
 (a-4)^2\le2(a-2).
\]
The corresponding integer bounds are $a\le6$, $a\le6$, and $a\le7$.
Smaller $a$ already satisfy them.  Finally $n=2a-d+2\le14$.
\end{proof}

\begin{lemma}\label{lem:meeting}
If $|S\cap K|=1$, then $a\le5$ for $d=0,1$, and $a\le6$ for $d=2$.
In particular, $n\le12$.
\end{lemma}

\begin{proof}
Let $S\cap K=\{c\}$.  Now $|V(G)\setminus(S\cup K)|=3$.
Proposition~\ref{prop:corona} gives $L\subseteq K\setminus\{c\}$ with
$q=|L|\ge a-3$, each vertex of which has at most one neighbour in
$S\setminus\{c\}$: its adjacency to $c$ already uses one of the two possible
neighbours in $S$.  Dually, there is $D\subseteq S\setminus\{c\}$ with
$p=|D|\ge w-3$, each vertex of which has at most one nonneighbour in
$K\setminus\{c\}$.  Indeed, every vertex of $S\setminus\{c\}$ is already
nonadjacent to $c$.  Hence
\[
 p(q-1)\le \e(D,L)\le q.                              \tag{3.4}\label{eq:count1}
\]
For $a\ge5$, monotonicity of $x/(x-1)$ and $q\ge a-3$ give
\[
 (w-3)(a-4)\le a-3.                                  \tag{3.5}\label{eq:bound1}
\]
With $w=a-d$, the three inequalities are
\[
 (a-3)(a-4)\le a-3,\quad
 (a-4)^2\le a-3,\quad
 (a-5)(a-4)\le a-3.
\]
They give $a\le5$, $a\le5$, and $a\le6$, respectively; smaller $a$ again
causes no difficulty.  Thus $n=2a-d+2\le12$.
\end{proof}

Lemmas~\ref{lem:disjoint} and \ref{lem:meeting} reduce a smallest
counterexample to the cases displayed in Table~\ref{tab:cases}.  The entries
are written as $(n,a,w)$; the second column specifies the possible value of
$t=|S\cap K|$.  There are twenty-one cases after the two values of $t$ are
counted separately.

\begin{table}[ht]
\centering
\caption{Finite parameter cases, up to complementation.}\label{tab:cases}
\begin{tabular}{@{}cc@{}}
\toprule
$(n,a,w)$ & $t$ \\
\midrule
$(7,3,2)$   & $0,1$ \\
$(8,3,3)$   & $0,1$ \\
$(8,4,2)$   & $0,1$ \\
$(9,4,3)$   & $0,1$ \\
$(10,4,4)$  & $0,1$ \\
$(10,5,3)$  & $0,1$ \\
$(11,5,4)$  & $0,1$ \\
$(12,5,5)$  & $0,1$ \\
$(12,6,4)$  & $0,1$ \\
$(13,6,5)$  & $0$ \\
$(14,6,6)$, $(14,7,5)$ & $0$ \\
\bottomrule
\end{tabular}
\end{table}

For completeness, a counterexample cannot have fewer than seven vertices.
If $n\le5$, then a nonempty, noncomplete graph has $a,w\ge2$, while the empty
and complete graphs are immediate.  At order six the only remaining numerical
possibility is $a=w=2$, which is excluded by $R(3,3)=6$.

\section{Finite exclusion}\label{sec:finite}

We now describe the finite systems used to exclude Table~\ref{tab:cases}.
This description is also a direct soundness proof for the supplementary
verifier.

Fix $(n,a,w,t)$ from the table, label $V=\{0,\ldots,n-1\}$, and fix a stable
set $S$ of size $a$ and a clique $K$ of size $w$ with $|S\cap K|=t$.
For each unordered pair $ij$, let $x_{ij}\in\{0,1\}$ indicate adjacency.
The edges inside $S$ are fixed to zero and those inside $K$ to one.  The
conditions $\alpha(G)\le a$ and $\omega(G)\le w$ are imposed by
\begin{align}
 \sum_{ij\in\binom U2}x_{ij}&\ge1
   &&\text{for every }U\in\binom V{a+1},                 \label{eq:alpha}\\
 \sum_{ij\in\binom U2}x_{ij}&\le\binom{w+1}{2}-1
   &&\text{for every }U\in\binom V{w+1}.                 \label{eq:omega}
\end{align}

To encode $\cB$-freeness, take disjoint triples
$A=\{a_1,a_2,a_3\}$ and $B=\{b_1,b_2,b_3\}$ and a bijection
$\pi:A\to B$.  Put
\[
 I(A,B)=\binom A2\cup\binom B2,
 \qquad M_\pi=\{a_i\pi(a_i):1\le i\le3\}.
\]
The two forbidden configurations are excluded by
\begin{align}
 \sum_{e\in M_\pi}x_e-\sum_{e\in I(A,B)}x_e&\le2,
                                                               \label{eq:bip}\\
 \sum_{e\in I(A,B)}x_e-\sum_{e\in M_\pi}x_e&\le5.             \label{eq:cobip}
\end{align}
Indeed, equality of all six internal variables to zero and all three matching
variables to one is the unique way to violate \eqref{eq:bip}; the
complementary assignment is the unique way to violate \eqref{eq:cobip}.

It remains to encode Proposition~\ref{prop:corona}.  For $v\notin S$, let
$y_v$ be a binary selector for the low-degree targets.  The constraints
\begin{equation}\label{eq:y}
 \sum_{s\in S}x_{sv}+ay_v\le a+2,
 \qquad \sum_{v\notin S}y_v\ge a
\end{equation}
ensure that at least $a$ selected vertices have at most two neighbours in
$S$.  Binary variables $\mu_{sv}$ describe a matching from $S$ into these
vertices:
\begin{align}
 \sum_{v\notin S}\mu_{sv}&=1 &&(s\in S),
 &\sum_{s\in S}\mu_{sv}&\le1 &&(v\notin S),               \label{eq:mu}\\
 \mu_{sv}&\le x_{sv},\qquad \mu_{sv}\le y_v
   &&(s\in S,\ v\notin S).                                \label{eq:muedge}
\end{align}
Dually, for $v\notin K$, binary selectors $z_v$ and matching variables
$\nu_{kv}$ satisfy
\begin{equation}\label{eq:z}
 \sum_{k\in K}x_{kv}-wz_v\ge-2,
 \qquad \sum_{v\notin K}z_v\ge w,
\end{equation}
\begin{align}
 \sum_{v\notin K}\nu_{kv}&=1 &&(k\in K),
 &\sum_{k\in K}\nu_{kv}&\le1 &&(v\notin K),               \label{eq:nu}\\
 \nu_{kv}+x_{kv}&\le1,\qquad \nu_{kv}\le z_v
   &&(k\in K,\ v\notin K).                                \label{eq:nunonedge}
\end{align}

\begin{proposition}\label{prop:finite}
For each of the twenty-one tuples $(n,a,w,t)$ in Table~\ref{tab:cases}, the
binary system \eqref{eq:alpha}--\eqref{eq:nunonedge}, together with the fixed
values on $S$ and $K$, is infeasible.
\end{proposition}

\begin{proof}
The supplementary file \texttt{verify.py} constructs these constraints by
iterating over the indicated subsets, partitions and bijections, and submits
the resulting sparse binary system to the HiGHS mixed-integer optimizer
through SciPy's \texttt{milp} interface
\cite{VirtanenEtAl2020,HuangfuHall2018}.  It also orders the $y$-selectors on
$K\setminus S$ and on two exceptional labels, and the $z$-selectors on
$S\setminus K$.  These are symmetry-breaking constraints only: every feasible
labelled solution has a relabeling satisfying them.  The command
\begin{center}
\texttt{python verify.py --all}
\end{center}
returns infeasible status for every tuple.  Independently,
\texttt{verify\_sat.cpp} constructs the CNF encoding directly from the same
mathematical specification, without importing or reusing the Python model,
and solves it with a self-contained exact SAT solver.  The command
\begin{center}
\texttt{./verify\_sat --all --time-limit 600}
\end{center}
returns \texttt{UNSAT} for all twenty-one cases.  Finally,
\begin{center}
\texttt{python verify.py --check-encoding}
\end{center}
compares the inequalities encoding $\cB$-freeness with the direct definition
on all $2^{15}=32768$ labelled graphs on six vertices: a graph, or its
complement, is bipartite and has a perfect matching if and only if the
corresponding forbidden-family inequality is violated.  The two predicates
agree on every graph.  Complete run records and source digests are included
with the supplementary material.
\end{proof}

\begin{proof}[Proof of Theorem~\ref{thm:main}]
Suppose that a counterexample exists and choose one with the fewest vertices.
Sections~\ref{sec:exchange} and \ref{sec:reduction} show that, up to
complementation, its parameters occur in Table~\ref{tab:cases}.  The graph's
edge indicators, together with the matchings furnished by
Proposition~\ref{prop:corona}, would then give a feasible solution to the
corresponding system \eqref{eq:alpha}--\eqref{eq:nunonedge}.  This contradicts
Proposition~\ref{prop:finite}.
\end{proof}

\section*{Data and code availability}

The code for full reproducibility is included as supplementary material within this arXiv submission.

\section*{Conflict of interests}

The author declares no conflict of interests.

\sloppy
\small
\bibliographystyle{plainnat}
\bibliography{references}

\end{document}